\documentclass[12pt]{article}
\usepackage{amsmath}
\usepackage{amssymb}
\usepackage{latexsym}
\usepackage{amsthm}
\usepackage{tabularx}
\usepackage{booktabs}
\usepackage{mathrsfs}
\usepackage{enumitem}
\usepackage{ytableau}
\usepackage{graphicx,colordvi}
\usepackage{algorithm,algorithmic}
\usepackage[colorlinks,
            linkcolor=red,
            anchorcolor=blue,
            citecolor=green]{hyperref}

\usepackage{epic}

\newtheorem{thm}{Theorem}[section]
\newtheorem{lem}[thm]{Lemma}
\newtheorem{prop}[thm]{Proposition}
\newtheorem{cor}[thm]{Corollary}
\newtheorem{conj}[thm]{Conjecture}
\newtheorem{rem}[thm]{Remark}

\allowdisplaybreaks

\begin{document}
\begin{center}
{\large \bf  The inverse Kazhdan-Lusztig polynomial of a matroid}
\end{center}

\begin{center}
Alice L.L. Gao$^1$ and Matthew H.Y. Xie$^2$\\[6pt]

School of Mathematics and Statistics,\\
Northwestern Polytechnical University, Xi'an, Shaanxi 710072, P.R. China

College of Science, \\
Tianjin University of Technology, Tianjin 300384, P. R. China\\[6pt]

Email: $^{1}${\tt llgao@nwpu.edu.cn},
      $^{2}${\tt xie@email.tjut.edu.cn}
\end{center}

\noindent\textbf{Abstract.}
In analogy with the classical Kazhdan-Lusztig polynomials for Coxeter groups, Elias, Proudfoot and Wakefield introduced the concept of Kazhdan-Lusztig polynomials for matroids. It is known that both the classical Kazhdan-Lusztig polynomials and the matroid Kazhdan-Lusztig polynomials can be considered as special cases of the Kazhdan-Lusztig-Stanley polynomials for locally finite posets. In the framework of Kazhdan-Lusztig-Stanley polynomials, we study the inverse of Kazhdan-Lusztig-Stanley functions and define the inverse Kazhdan-Lusztig polynomials for matroids. We also compute these polynomials for boolean matroids and uniform matroids. As an unexpected application of the inverse Kazhdan-Lusztig polynomials, we obtain a new formula to compute the Kazhdan-Lusztig polynomials for uniform matroids. Similar to the Kazhdan-Lusztig polynomial of a matroid, we conjecture that the coefficients of its inverse Kazhdan-Lusztig polynomial are nonnegative and log-concave.

\noindent \emph{AMS Classification 2020:} 05B35


\section{Introduction}

In the study of the Hecke algebra of Coxeter groups, Kazhdan and Lusztig \cite{Kazhh-Lusztig}
associated to each pair of group elements an integral polynomial, now known as the
Kazhdan-Lusztig polynomial. In \cite{Inverse-KL-AMA1980} Kazhdan and Lusztig  associated another integral polynomial
to each pair of group elements, now known as the inverse Kazhdan-Lusztig polynomial.
In analogy with the Kazhdan-Lusztig polynomial for Coxeter groups,
Elias, Proudfoot and Wakefield \cite{ElisKL2016adv} introduced the Kazhdan-Lusztig polynomials for matroids, which can also be defined for each pair of comparable elements in the lattice of flats. For more information on the matroid Kazhdan-Lusztig polynomials, we refer the reader to \cite{Gedeon2017ejc, Gedeon_proud_2017sem,lxy_2018arxiv,glxyz_2018arxiv}. It is natural to consider whether one can define the inverse Kazhdan-Lusztig polynomials for matroids. The aim of this paper is twofold: theoretical existence of the inverse Kazhdan-Lusztig polynomials for matroids, and practical application of these polynomials for the computation of the matroid Kazhdan-Lusztig polynomial.

Let us first recall the definition of matroid Kazhdan-Lusztig polynomials. We will adopt the same notations as in of Elias, Proudfoot and Wakefield \cite{ElisKL2016adv}. We assume that the reads are familiar with the concept of matroids; see \cite{Oxley2011Oxfordmatroid} for more information. Given a matroid $M$, let $L(M)$ denote the lattice of flats of $M$ and let $\chi_M(t)$ denote its characteristic polynomial. For any flat $F\in L(M)$,  let $M^F$ be the contraction of $M$ at $F$ and let $M_F$ be the localization of $M$ at $F$. Moreover, let $\mathrm{rk}\, M$ denote the rank of $M$. Elias, Proudfoot and Wakefield
proved the following result.

\begin{thm}[{\cite[Theorem 2.2]{ElisKL2016adv}}]\label{KL-prop}
	There is a unique way to assign to each matroid $M$ an element $P_M(t) \in \mathbb{Z}[t]$ such that the following conditions are satisfied:
\begin{itemize}
\item[($i$).] If $\mathrm{rk} \,M=0$, then $P_M(t)=1$.
\item[($ii$).] If $\mathrm{rk} \,M>0$, then $\mathrm{deg} \, P_M(t)< \frac{1}{2}   \mathrm{rk} \, M$.
\item[($iii$).] For every matroid $M$, we have
\begin{align}\label{KL_matroid-epw}
t^{\mathrm{rk}\, M}P_M(t^{-1})=\sum_{F\in L(M)}\chi_{M_F}(t)P_{M^F}(t).
\end{align}
\end{itemize}	
\end{thm}

Given a matroid $M$, the uniquely defined polynomial $P_M(t)$ in Theorem \ref{KL-prop} is called the Kazhdan-Lusztig polynomial of $M$.
In this paper we will associate $M$ another integral polynomial  $Q_M(t)$, called the inverse Kazhdan-Lusztig polynomial of $M$, whose existence is guaranteed by the following theorem.

\begin{thm}\label{inverseKL-prop}
	There is a unique way to assign to each matroid $M$ an element $Q_M(t) \in \mathbb{Z}[t]$ such that the following conditions are satisfied:
\begin{itemize}
\item[($i^{\prime}$).] If $\mathrm{rk} \,M=0$, then $Q_M(t)=1$.
\item[($ii^{\prime}$).] If $\mathrm{rk} \,M>0$, then $\mathrm{deg} \, Q_M(t)< \frac{1}{2}   \mathrm{rk} \, M$.
\item[($iii^{\prime}$).] For every matroid $M$, we have
\begin{align}\label{inverse_KL_matroid-2}
t^{\mathrm{rk} \, M} \cdot  (-1)^{\mathrm{rk} \, M} Q_M(t^{-1})=\sum_{F\in L(M)}
 (-1)^{\mathrm{rk} \, M_{F}} Q_{M_F}(t)     \cdot t ^{\mathrm{rk} \, M^F}  \chi_{M^F}(t^{-1}).
\end{align}
\end{itemize}	
\end{thm}


Comparing Theorem \ref{KL-prop} and Theorem \ref{inverseKL-prop}, it seems that $P_{M}(t)$ and $Q_{M}(t)$ do not make any difference formally. However, these two polynomials will play different roles for exploring the properties of matroids. To make it better to understand this point, we will deal with the matroid Kazhdan-Lusztig polynomials and the inverse Kazhdan-Lusztig polynomials under the general framework of Kazhdan-Lusztig-Stanley polynomials for locally finite posets; see \cite{stanley1992kls,brenti1999kls1,brenti1999kls2,Dye93} for more information. In fact, both the classical Kazhdan-Lusztig polynomials and the Kazhdan-Lusztig polynomials of matroids can be considered  as  special cases of the Kazhdan-Lusztig-Stanley polynomials, see Gedeon, Proudfoot and Young \cite{Gedeon_proud_2017sem}, Wakefield \cite{Wakefield_elec_2018} and Proudfoot \cite{proudfoot2017arxiv_alge-geo}.

It is known that the matroid Kazhdan-Lusztig polynomials posses many conjectured interesting properties. It would be desirable if one can obtain an explicit formula for $P_{M}(t)$. However, the computation of $P_{M}(t)$ is full of challenge, even for simple matroids. For example, the conjecture on the leading coefficients of the Kazhdan-Lusztig polynomials for braid matriods is still open, see \cite{ElisKL2016adv} and \cite{Gedeon_proud_2017sem}. Some progress has been made for computing the matroid Kazhdan-Lusztig polynomials. Gedeon, Proudfoot and Young \cite{Geden2017jcta1} introduced the concept of equivariant Kazhdan-Lusztig polynomials, which proved to be very useful for computing the Kazhdan-Lusztig polynomials for uniform matroids. Lu, Xie and Yang \cite{lxy_2018arxiv} used the method of generating functions to obtain explicit formulas for the Kazhdan-Lusztig polynomials of fan matroids, wheel matroids and whirl matroids. Lee, Nasr and Radcliffe studied the Kazhdan-Lusztig polynomials of $\rho$-removed uniform matroids in \cite{LeeNasrRadcliffe_arxiv_1} and sparse paving matroids in  \cite{LeeNasrRadcliffe_arxiv_2}. By using the concept of $Z$-polynomials, Proudfoot, Xu and Young \cite{Proudfoot_Xu_Young_elc_2018} obtained a faster algorithm for computing the Kazhdan-Lusztig polynomials for braid matroids.
Braden and Vysogorets \cite{Branden_Vysoreet_elec_2020}  gave a new recursive method to compute the matroid Kazhdan-Lusztig polynomials, which was particularly useful for parallel connection matroids, including double-cycle graphs, partial saw graphs and fan graphs.

We would like to point out that the inverse Kazhdan-Lusztig polynomials introduced here are also useful for the computation of the matriod Kazhdan-Lusztig polynomials. We will use the following relations between these two families of polynomials.

\begin{thm}\label{KL-inverse-KL-2}
Given a matroid $M$ of positive rank, let $E$ be the ground set of $M$. Then
\begin{align}
P_M(t)&=-\sum_{E \neq F\in L(M)}P_{M_F}(t) \cdot (-1) ^{\mathrm{rk} \, M^F} Q_{M^F}(t),\label{from_Q_to_P}\\[5pt]
P_{M}(t)&=-\sum_{\emptyset \neq F\in L(M)} (-1)^{\mathrm{rk} \, M_{F}}Q_{M_F}(t)\cdot P_{M^F}(t),\\[5pt]
Q_M(t)&=-\sum_{\emptyset \neq F\in L(M)}(-1)^{\mathrm{rk} \, M_{F}} P_{M_F}(t) \cdot Q_{M^F}(t),\label{from_P_to_Q}\\[5pt]
Q_M(t)&=-\sum_{E \neq F\in L(M)} Q_{M_F}(t)\cdot (-1) ^{\mathrm{rk} \, M^F} P_{M^F}(t).
\end{align}
\end{thm}

The reason for this applicability is so that the above relations provide more flexibility. For instance, for some matroid $M$ and any flat $F\neq E$ the polynomial $P_{M_F}(t)$ in \eqref{from_Q_to_P} can be easily computed, and moreover $Q_{M^F}(t)$ can be easily computed by using \eqref{inverse_KL_matroid-2}.
In this paper we will take the uniform matroids as examples to illustrate this application.
In fact, our method reduces the computation of the Kazhdan-Lusztig polynomials and the inverse Kazhdan-Lusztig polynomials for uniform matroids to that for boolean matroids, which turns out to be very simple.

This paper is organized as follows. In Section  \ref{sec-2} we will prove Theorems \ref{inverseKL-prop} and \ref{KL-inverse-KL-2} from the viewpoint of Kazhdan-Lusztig-Stanley polynomials. In Section \ref{sec-3} we will compute the inverse Kazhdan-Lusztig polynomials for boolean matroids and uniform matroids. In particular, we derive a new formula for computing the Kazhdan-Lusztig polynomial for uniform matroids. In Section \ref{sec-4}, we will propose some conjectures on the inverse Kazhdan-Lusztig polynomials for further study.

\section{The inverse Kazhdan-Lusztig polynomial}\label{sec-2}

In this section we aim to prove Theorem \ref{inverseKL-prop} and Theorem \ref{KL-inverse-KL-2}. To this end, we will first follow Proudfoot \cite{proudfoot2017arxiv_alge-geo} to give an overview of the theory of Kazhdan-Lusztig-Stanley polynomials for general locally finite posets, and then apply the related theory to the lattices of flats of matroids to prove our main results.

\subsection{The Kazhdan-Lusztig-Stanley polynomial}

This subsection aims to review the general theory of Kazhdan-Lusztig-Stanley polynomials.
We adopt the notations of \cite{proudfoot2017arxiv_alge-geo}.  Given a poset $P$, we say that $P$ is locally finite, if for all $x\leq z \in P$, the set
$$[x,z]:=\{y \in P | x \leq y \leq  z \}$$
is finite. Here we often say that $[x,z]$ is an  interval of $P$.
From now on we assume that $P$ is always locally finite.
Let $\mathrm{Int}(P)$ denote the set of all intervals of $P$.
Then let  $I(P)$ be the set of all functions
$$f: \mathrm{Int}(P)   \rightarrow \bigoplus_{[x,z]\in \mathrm{Int}(P)}\mathbb{Z}[t],$$
where $f$ is defined by letting
$f([x,z]) \in \mathbb{Z}[t]$ for any $[x,z]\in \mathrm{Int}(P)$.
For any $f \in I(P)$ and $[x,z]\in \mathrm{Int}(P)$, we often write $f([x,z])$ as  $f_{xz}(t)$ for shortage, and say that  $f_{xz}(t)$ is  the  component of $f$ at   interval $[x,z]$.
It is known that $I(P)$ admits a ring structure with product given by the convolution:
\begin{align}\label{eq-conv-product}
(fg)_{xz}(t):=\sum_{x\leq y \leq z}f_{xy}(t)\cdot g_{yz}(t)
\end{align}
for any $f,g \in I(P)$ and any  interval $[x,z]\in \mathrm{Int}(P)$.
Moreover, the identity element in $I(P)$ is the function $\delta$ with the property that $\delta_{xz} = 1$ if $x=z$ and 0 otherwise.

We say that $P$ is a weakly ranked poset if it is locally finite and there exists a function $r \in I(P)$, called a weak rank function, such that it satisfies the following conditions:
\begin{itemize}
\item
$r_{xy}\in \mathbb{Z}$ for all $x\leq y \in P$,
\item
if $x<y$, then $r_{xy}>0$,
\item if $x\leq y \leq z$, then $r_{xy}+r_{yz}=r_{xz}$.
\end{itemize}

For any weakly ranked poset $P$, let $\mathscr{I}(P) \subset I(P)$ denote the subring of functions $f$ with the property that
$\mathrm{deg}~f_{xz}(t) \leq   r_{xz} $
  for all $x\leq z \in P$.
The ring $\mathscr{I}(P)$ admits an involution $f \longrightarrow \overline{f}$ defined by the formula
$$\overline{f}_{xz}(t):=t^{r_{xz}}f_{xz}(t^{-1}).$$
Moreover, let   $\mathscr{I}_{1/2}(P) \subset \mathscr{I}(P)$ denote the set of functions of $f$ with the property that
$ f_{xx}(t)=1$ for all $ x\in P$ and $\mathrm{deg}~f_{xz}(t)< \frac{1}{2} \, r_{xz} $ for all $ x<z \in P$.

It is known that
an element $f \in I(P)$ has an inverse (left or right) if and only if $f_{xx}(t)=\pm 1$ for all $x\in P$;  see \cite[Lemma 2.1 ]{proudfoot2017arxiv_alge-geo}. In this case, the left and right inverses are unique and they coincide.
Moreover, if $f \in \mathscr{I}(P)$ is invertible, then  $f^{-1} \in \mathscr{I} (P )$. An element $\kappa \in \mathscr{I}(P)$ is called   a $P$-kernel if $\kappa_{xx}(t)=1$  for all $x\in P$ and $\overline{\kappa}=\kappa^{-1}$.
We have the following result, as stated in \cite[Theorem 2.2]{proudfoot2017arxiv_alge-geo}.

\begin{thm}[{\cite[Corollary 6.7]{stanley1992kls}, \cite[Proposition 1.2]{Dye93} and \cite[Theorem 6.2]{brenti1999kls1}}\label{right-left-KLS}] Suppose that $P$ is a weakly ranked poset and $\kappa \in \mathscr{I}(P)$ is a $P$-kernel. Then there exists a unique pair of functions $f,g \in \mathscr{I}_{1/2}(P)$ such that
\begin{align}\label{eq-right-left-kls}
\overline{f}=\kappa f \qquad \mbox{and} \qquad \overline{g}=g \kappa.
\end{align}
\end{thm}

We will refer to $f$ as the right Kazhdan-Lusztig-Stanley function associated with $\kappa$, and  $g$ as the left Kazhdan-Lusztig-Stanley function associated with $\kappa$. For any $x\leq z\in P$, we will refer to the polynomial $f_{xz}(t)$ as a right Kazhdan-Lusztig-Stanley polynomial, and the polynomial $g_{xz}(t)$ as a left Kazhdan-Lusztig-Stanley polynomial.
Since $f,g\in \mathscr{I}_{1/2}(P)$, we see that $f_{xx}(t)=g_{xx}(t)=1$ for all $x \in P$ and hence both $f$ and $g$ are invertible.
It is obvious that $f,g\in \mathscr{I}_{1/2}(P)$, see \cite[Section 4]{brenti19988ASPM}. From \eqref{eq-right-left-kls} it immediately follows that
\begin{align}\label{Q_relation_bar}
\overline{f^{-1}}=f^{-1}\overline{\kappa}
\mathrm{~~~~and~~~}
\overline{g^{-1}}=\overline{\kappa}g^{-1}.
\end{align}
This means that $f^{-1}$ is the left Kazhdan-Lusztig-Stanley function associated with $P$-kernel $\overline{\kappa}$, and  $g^{-1}$ is the corresponding right Kazhdan-Lusztig-Stanley function.

 \subsection{The inverse Kazhdan-Lusztig polynomial of a matroid}

 This subsection aims to apply the theory of Kazhdan-Lusztig-Stanley polynomials to
 prove Theorem \ref{inverseKL-prop} and Theorem \ref{KL-inverse-KL-2}. To this end, we take $P$ in the above subsection to be
 the lattice $L(M)$ of flats of a matroid $M$ with ground set $E$.

 It is not difficult to see that $L(M)$ is a weakly ranked poset if it is equipped with a function $r \in I(L(M))$,
 which assigns $\mathrm{rk}\, M_G-\mathrm{rk}\, M_F$ to each interval $[F, G]$ for any $F\leq G \in L(M)$. Define $\zeta \in  {I}(L(M))$ to be the function given by $\zeta_{F,G}(t)= 1$ for all $F \leq G\in L(M)$. Clearly, $\zeta \in  \mathscr{I}_{1/2}(L(M))$, which is invertible. The element $$\mu:=\zeta^{-1} \in  \mathscr{I}_{1/2}(L(M))$$
 coincides with the M$\mathrm{\ddot{o}}$bius function of $L(M)$.
Since both $\zeta$ and $\mu$ belong to $\mathscr{I}(L(M))$, using the involution on $\mathscr{I}(L(M))$ one can define
the product
$$\chi:=\mu \overline{\zeta} =\zeta^{-1}\overline{\zeta},$$
which is called the characteristic function of $L(M)$. It is known that $\chi_{\emptyset, E}(t)$ coincides with the characteristic polynomial of $\chi_{M}(t)$. Note that
$$\chi^{-1}=(\overline{\zeta})^{-1}\zeta=\overline{\zeta^{-1}}\zeta=\overline{\chi}.$$
Thus $\chi$ is a $L(M)$-kernel. By \eqref{eq-right-left-kls} and \eqref{Q_relation_bar} there exists a unique pair of functions $f,g \in \mathscr{I}_{1/2}(L(M))$ such that
\begin{align}\label{eq-key-identity_2}
\overline{f^{-1}}=f^{-1}\overline{\chi}
\mathrm{~~~~and~~~}
\overline{g^{-1}}=\overline{\chi}g^{-1}.
\end{align}

\begin{rem}
  
For any pairs of flats $F \leq G \in L(M)$,
let $M^*$ denote the matroid with
 $L(M^*)\cong [F,G]$ and  ground set $E^*$. As discussed before, one can define $L(M^*)$-kernel $\chi^*$. For $\chi^*$ there exists a unique pair of functions
 $f^*,g^* \in \mathscr{I}_{1/2}(L(M^*))$ such that
\begin{align}\label{eq-key-identity}
\overline{(f^*)^{-1}}=(f^*)^{-1}\overline{\chi^*}
\mathrm{~~~~and~~~}
\overline{(g^*)^{-1}}=\overline{\chi^*}(g^*)^{-1}.
\end{align}
Due to the invariant property of the characteristic function for matroids, we have
$$(\chi^{-1})_{F, G}(t)=((\chi^*)^{-1})_{\emptyset, E^*}(t),$$
and therefore
\begin{align}\label{eq-invariant-identity}
(f^{-1})_{F, G}(t)=((f^*)^{-1})_{\emptyset, E^*}(t)
\end{align}
since the left and right Kazhdan-Lusztig-Stanley functions are uniquely determined by the kernel for the lattice of flats. 
\end{rem}

We are now in the position to prove Theorem \ref{inverseKL-prop}.

\proof [Proof of Theorem \ref{inverseKL-prop}]
To prove the existence, for any matroid $M$ let
$$\Hat{Q}_M(t):=(f^{-1})_{\emptyset, E}(t)    \mathrm{~~~~and~~~}     Q_M(t):=(-1)^{\mathrm{rk} \, M}\Hat{Q}_M(t).$$
We proceed to show that $Q_M(t)$ satisfies the conditions ($i^{\prime}$), ($ii^{\prime}$) and ($iii^{\prime}$) in Theorem \ref{inverseKL-prop}.
If $\mathrm{rk}~M=0$, whence $\emptyset=E$, then $$Q_M(t)=(-1)^{\mathrm{rk}~M}(f^{-1})_{\emptyset,E}(t)=(f^{-1})_{\emptyset,\emptyset}(t)=1$$
since $f^{-1}\in \mathscr{I}_{1/2}( L(M))$, and hence ($i^{\prime}$) holds. 
If $\mathrm{rk} \, M> 0$, whence $\emptyset <E$, then
$$\mathrm{deg}~Q_M(t)=\mathrm{deg}~(f^{-1})_{\emptyset, E}(t)<\frac{1}{2} \, r_{\emptyset, E}=\frac{1}{2}~\mathrm{rk} \, M$$
again by $f^{-1}\in \mathscr{I}_{1/2}( L(M))$, and therefore condition ($ii^{\prime}$) holds. 
By \eqref{eq-key-identity_2}, we have
\begin{align}\label{eq-leftright-keyid} 
\overline{f^{-1}}_{\emptyset, E}(t)=(f^{-1}\overline{\chi})_{\emptyset, E}(t).
\end{align}
For the left hand side we find that
$$(\overline{f^{-1}})_{\emptyset, E}(t)=t^{r_{\emptyset,E}}(f^{-1})_{\emptyset, E}(t^{-1})
=t^{\mathrm{rk} \, M} \Hat{Q}_M(t^{-1})=t^{\mathrm{rk} \, M} \cdot  (-1)^{\mathrm{rk} \, M} Q_M(t^{-1}).$$
By \eqref{eq-conv-product} and \eqref{eq-invariant-identity}, the right hand side of \eqref{eq-leftright-keyid} becomes
\begin{align*}
(f^{-1}\overline{\chi})_{\emptyset, E}(t)
=&\sum_{F\in L(M)}
   (f^{-1})_{\emptyset, F}(t)  \cdot   t^{r_{F, E}}\chi_{F, E}(t^{-1})= \sum_{x\in P}
 \Hat{Q}_{M_F}(t)     \cdot t ^{\mathrm{rk} \, M^F}  \chi_{M^F}(t^{-1})\\[5pt]
 =&\sum_{F\in L(M)}
 (-1)^{\mathrm{rk} \, M_{F}} Q_{M_F}(t)     \cdot t ^{\mathrm{rk} \, M^F}  \chi_{M^F}(t^{-1}).
\end{align*}
Combining the above two identities, we establish the validity of ($iii^{\prime}$).

It remains to show the uniqueness of $Q_M(t)$. To this end, we rewrite \eqref{inverse_KL_matroid-2} as
\begin{align*}
t^{\mathrm{rk} \, M} Q_M(t^{-1})-  Q_{M}(t)
=\sum_{E \neq F\in L(M)}
 (-1)^{\mathrm{rk} \, M_{F}} Q_{M_F}(t)     \cdot (-t) ^{\mathrm{rk} \, M^F}  \chi_{M^F}(t^{-1}).
\end{align*}
By induction on the rank of $M$, we may assume that the right hand side of the above equation has been fixed. 
It is clear that there can be at most one polynomial $Q_M(t)$ of degree strictly less than $\frac{1}{2}~ \mathrm{rk} \, M$ satisfying this condition, as desired.
\qed

%

\begin{rem}
The above proof shows that \eqref{inverse_KL_matroid-2} is equivalent to \eqref{eq-key-identity_2}.
Taking the involution on both sides of \eqref{eq-key-identity_2}, we obtain 
\begin{align*}
f^{-1}=\overline{f^{-1}}\chi,
\end{align*}
from which we get that
\begin{align*}
(f^{-1})_{\emptyset, E }(t)
=\sum_{F\in L(M)}
   t^{r_{\emptyset, F }}(f^{-1})_{\emptyset, F}(t^{-1}) \cdot   \chi_{F, E}(t).
\end{align*}
It follows that
\begin{align}\label{inv_KL_equa_2}
\Hat{Q}_M(t)= \sum_{F\in L(M)}
 t^{\mathrm{rk} \, M_{F}} \Hat{Q}_{M_F}(t^{-1}) \cdot  \chi_{M^F}(t),
\end{align}
which is equivalent to
\begin{align}\label{inv_KL_equa_2_2}
Q_M(t)=(-1)^{\mathrm{rk} \, M} \sum_{F\in L(M)}
 (-t)^{\mathrm{rk} \, M_{F}} Q_{M_F}(t^{-1}) \cdot     \chi_{M^F}(t)
\end{align}
For the computational convenience of the inverse Kazhdan-Lusztig polynomials for uniform matroids, we will use \eqref{inv_KL_equa_2} rather than \eqref{inv_KL_equa_2_2}.
\end{rem}

The reason that we call $Q_M(t)$ the inverse Kazhdan-Lusztig polynomial of $M$ is now clear. From the above proof we see that
$Q_M(t):=(-1)^{\mathrm{rk} \, M}(f^{-1})_{\emptyset, E}(t)$. While for the Kazhdan-Lusztig polynomial of $M$ we have $P_{M}(t)=f_{\emptyset, E}(t)$, see \cite{proudfoot2017arxiv_alge-geo}. We proceed to prove Theorem \ref{KL-inverse-KL-2}. 

\proof[Proof of Theorem \ref{KL-inverse-KL-2}]
Note that $E\neq \emptyset$ and $ff^{-1}=\delta$.
Thus $(ff^{-1})_{\emptyset, E}=0$, that is, 
$$\sum_{F \in L(M)}f_{\emptyset, F}(t) \cdot (f^{-1})_{F, E}(t)=0,$$
which, in terms of $P_M(t)$ and $Q_M(t)$, can be rewritten as
\begin{align}\label{relation_KL-IKL}
\sum_{F\in L(M)}P_{M_F}(t) \cdot  \Hat{Q}_{M^F}(t)=0.
\end{align}
When $F=E$, we have $P_{M_F}(t)=P_{M}(t)$ and $\Hat{Q}_{M^F}(t)=1$, from which it follows that
$$P_{M}(t)=-\sum_{E \neq F\in L(M)}P_{M_F}(t) \cdot  \Hat{Q}_{M^F}(t)
=-\sum_{E \neq F\in L(M)}P_{M_F}(t) \cdot (-1)^{\mathrm{rk} \, M^{F}}Q_{M^F}(t),$$
that is \eqref{from_Q_to_P}. By taking $F=\emptyset$ in \eqref{relation_KL-IKL}, we will obtain \eqref{from_P_to_Q}.  
The other two identities can be proved in the same manner by using the relation $f^{-1}f=\delta$. 
\qed

Recall that the constant term of the Kazhdan-Lusztig polynomial $P_M(t)$ of a matroid is always $1$. For the inverse Kazhdan-Lusztig polynomial $Q_M(t)$, we have the following result. 

\begin{prop}
For any matroid $M$, the constant term of $Q_M(t)$ is equal to the constant term of $\chi_M(t)$.
\end{prop}

\proof We use $[t^0]H(t)$ denote the constant term of a polynomial $H(t)$. 
Taking the constant term of both sides of \eqref{inv_KL_equa_2}, we obtain
$$[t^0]\Hat{Q}_M(t)=[t^0]\sum_{F\in L(M)}
 t ^{\mathrm{rk} \, M_F} \Hat{Q}_{M_F}(t^{-1}) \cdot     \chi_{M^F}(t) .$$
According to the definition of   inverse Kazhdan-Lusztig polynomial, if $\mathrm{rk} \, M_F\geq 1$,  then
 $\mathrm{deg} \, \Hat{Q}_{M_F}(t)< \frac{1}{2}   \mathrm{rk} \, M_F$.
 A little thought shows that
 the degree of the lowest term in $ t ^{\mathrm{rk} \, M_F} \Hat{Q}_{M_F}(t^{-1})$  is strictly greater than $\frac{1}{2}   \mathrm{rk} \, M_F$.
Recall that
 $\chi_{M^F}(t) \in \mathbb{Z}[t]$.
We find that the constant term in $t ^{\mathrm{rk} \, F} \Hat{Q}_{M_F}(t^{-1}) \cdot     \chi_{M^F}(t) $ is $0$ if $\mathrm{rk} \, M_F\geq 1$. 
  If $\mathrm{rk} \, M_F=0$, it is clearly that
 $F=\emptyset$. Thus,
 $t ^{\mathrm{rk} \, F} \Hat{Q}_{M_F}(t^{-1}) \cdot     \chi_{M^F}(t) = \chi_{M}(t) .$
 Therefore, we have
 $$[t^0]\Hat{Q}_M(t)=[t^0]\chi_M(t).$$ This completes the proof. 
 \qed
 
\section{Special matriods}\label{sec-3}

The aim of this section is to compute the inverse Kazhdan-Lusztig polynomials for boolean matroids and uniform matroids.
In this section we will also show how to compute the Kazhdan-Lusztig polynomial for uniform matroids by using the inverse Kazhdan-Lusztig polynomials.

\subsection{Boolean matroids}

In order to compute the inverse Kazhdan-Lusztig polynomials for boolean matroids, let us first prove that the inverse Kazhdan-Lusztig polynomial for a matroid is multiplicative on direct sums. It is well known that the characteristic polynomial is  multiplicative on direct sums, namely,
\begin{align}\label{chr-mul}
\chi_{M_1\oplus M_2}(t)=\chi_{M_1}(t)\chi_{ M_2}(t),
\end{align}
see \cite[p.121]{White_Neil}.
Recall that Elias, Proudfoot and Wakefield already showed that the  Kazhdan-Lusztig polynomial is also multiplicative on direct sums, see \cite[Proposition 2.7]{ElisKL2016adv}.

\begin{lem}\label{sum}
For any matroids $M_1$ and $M_2$ with respective ground set $E_1$ and $E_2$,  we have
$$Q_{M_1\oplus M_2}(t)=Q_{M_1}(t)Q_{ M_2}(t).$$
\end{lem}

\proof One can prove the assertion in the same manner as Elias, Proudfoot and Wakefield did for the matroid Kazhdan-Lusztig polynomials. Recalling $Q_M(t)=(-1)^{\mathrm{rk} \, M}\Hat{Q}_M(t)$, it suffices to show that
$$\Hat{Q}_{M_1\oplus M_2}(t)=\Hat{Q}_{M_1}(t)\Hat{Q}_{ M_2}(t).$$
We use induction on the sum $\mathrm{rk}\, M_1+\mathrm{rk}\, M_2$. The statement is clear when $\mathrm{rk}\, M_1=0$ or $\mathrm{rk}\, M_2=0$. Now assume that the statement holds for $M_1'$ and $M_2'$ whenever one of $\mathrm{rk}\, M_1'\leq  \mathrm{rk}\, M_1$ and $\mathrm{rk}\, M_2'\leq  \mathrm{rk}\, M_2$ is strict.

By the properties of flats we have
$$L(M_1\oplus M_2)=L(M_1)\times L(M_2),$$
and
$$\mathrm{rk} (M_1\oplus M_2)_{(F_1,F_2)}=\mathrm{rk}\, (M_1)_{F_1}+\mathrm{rk}\, (M_2)_{F_2}$$
for any $(F_1,F_2)\in L(M_1\oplus M_2)$. Moreover,
\begin{align*}
  (M_1\oplus M_2)_{(F_1,F_2)} & = (M_1)_{F_1} \oplus (M_2)_{F_2},\\[6pt]
  (M_1\oplus M_2)^{(F_1,F_2)} & = (M_1)^{F_1} \oplus (M_2)^{F_2}.
\end{align*}
Thus, from \eqref{inv_KL_equa_2} and \eqref{chr-mul} it follows that
\begin{align*}
	\Hat{Q}_{M_1\oplus M_2}(t)&=\sum_{(F_1,F_2)\in L(M_1\oplus M_2)}t^{\mathrm{rk}\, (M_1\oplus M_2)_{(F_1,F_2)}}\Hat{Q}_{(M_1\oplus M_2)_{(F_1,F_2)}}(t^{-1}) \chi_{(M_1\oplus M_2)^{(F_1,F_2)}}(t)\\[6pt]
	&=\sum_{(F_1,F_2)\in L(M_1\oplus M_2)}t^{\mathrm{rk}\, (M_1)_{F_1}+\mathrm{rk}\, (M_2)_{F_2}}\Hat{Q}_{(M_1)_{F_1} \oplus (M_2)_{F_2}}(t^{-1}) \chi_{(M_1)^{F_1} \oplus (M_2)^{F_2}}(t).
\end{align*}
By induction, we get that
\begin{align}
&\Hat{Q}_{M_1\oplus M_2}(t)=t^{\mathrm{rk}\, (M_1)_{E_1}+\mathrm{rk}\, (M_2)_{E_2}}\Hat{Q}_{(M_1)_{E_1} \oplus (M_2)_{E_2}}(t^{-1}) \chi_{(M_1)^{E_1} \oplus (M_2)^{E_2}}(t)\nonumber\\[5pt]
&\quad+\sum_{(F_1,F_2)\neq (E_1, E_2)}t^{\mathrm{rk}\, (M_1)_{F_1}+\mathrm{rk}\, (M_2)_{F_2}}\Hat{Q}_{(M_1)_{F_1} \oplus (M_2)_{F_2}}(t^{-1}) \chi_{(M_1)^{F_1} \oplus (M_2)^{F_2}}(t)\nonumber\\[10pt]
=&t^{\mathrm{rk}\, M_1+\mathrm{rk}\, M_2}\Hat{Q}_{M_1 \oplus M_2}(t^{-1})\nonumber\\[5pt]
&\quad+\sum_{(F_1,F_2)\neq (E_1, E_2)}\left(t^{\mathrm{rk}\, (M_1)_{F_1}}Q^{*}_{(M_1)_{F_1}}(t^{-1})\chi_{(M_1)^{F_1}}(t)\right)\cdot \left(t^{\mathrm{rk}\, (M_2)_{F_2}}Q^{*}_{(M_2)_{F_2}}(t^{-1})\chi_{(M_2)^{F_2}}(t)\right).\label{eq-before}
\end{align}
By \eqref{inv_KL_equa_2} we have
\begin{align*}
\Hat{Q}_{M_1}(t)&= \sum_{F_1\in L(M_1)}
 t^{\mathrm{rk} \, (M_1)_{F_1}} \Hat{Q}_{(M_1)_{F_1}}(t^{-1}) \cdot     \chi_{(M_1)^{F_1}}(t),\\
\Hat{Q}_{M_2}(t)&= \sum_{F_2\in L(M_2)}
 t^{\mathrm{rk} \, (M_2)_{F_2}} \Hat{Q}_{(M_2)_{F_2}}(t^{-1}) \cdot     \chi_{(M_2)^{F_2}}(t),
\end{align*}
which leads to
\begin{align}
&\Hat{Q}_{M_1}(t)\Hat{Q}_{M_2}(t)
=t^{\mathrm{rk}\, M_1}\Hat{Q}_{M_1}(t^{-1})\cdot t^{\mathrm{rk}\, M_2}\Hat{Q}_{M_2}(t^{-1})\nonumber\\[5pt]
&\quad+\sum_{(F_1,F_2)\neq (E_1, E_2)}\left(t^{\mathrm{rk}\, (M_1)_{F_1}}Q^{*}_{(M_1)_{F_1}}(t^{-1})\chi_{(M_1)^{F_1}}(t)\right)\cdot \left(t^{\mathrm{rk}\, (M_2)_{F_2}}Q^{*}_{(M_2)_{F_2}}(t^{-1})\chi_{(M_2)^{F_2}}(t)\right).\label{eq-after}
\end{align}
Substracting \eqref{eq-after} from \eqref{eq-before} yields
\begin{align*}
\Hat{Q}_{M_1\oplus M_2}(t)-\Hat{Q}_{M_1}(t)\Hat{Q}_{M_2}(t)=	
t^{\mathrm{rk}\, M_1+\mathrm{rk}\, M_2}\Hat{Q}_{M_1\oplus M_2}(t^{-1}) -t^{\mathrm{rk}\, M_1}\Hat{Q}_{M_1}(t^{-1})\cdot t^{\mathrm{rk}\, M_2}\Hat{Q}_{M_2}(t^{-1}).
\end{align*}
We need to show that $$\Hat{Q}_{M_1\oplus M_2}(t)-\Hat{Q}_{M_1}(t)\Hat{Q}_{M_2}(t)=0.$$
Suppose the contrary. Then, as a polynomial in $t$, the degree of the left-hand side is strictly less than $\frac{1}{2}(\mathrm{rk}\, M_1+\mathrm{rk}\, M_2)$, while the lowest degree of the right-hand side is strictly greater than $\frac{1}{2}(\mathrm{rk}\, M_1+\mathrm{rk}\, M_2)$, a contradiction. This completes the proof.
\qed

We proceed to compute the inverse Kazhdan-Lusztig polynomial for boolean matroids. Given $n\geq 1$, let $B_n$ denote the boolean matroid of rank $n$. The following result is immediate.

\begin{cor}\label{Bollean_IKL}
	For any boolean matroid $B_n$ of positive rank, we have  $$Q_{B_n}(t)=1.$$
\end{cor}
\begin{proof}
From \eqref{inverse_KL_matroid-2} it is routine to compute that
$Q_{B_1}(t)=1.$
Note that $B_n$ is isomorphic to the direct sum of $n$ copies of $B_1$.
The desired result immediately follows from Lemma \ref{sum}.
\end{proof}


\subsection{Uniform matroids}
This section aims to give a general formula for the  inverse Kazhdan-Lusztig polynomial of uniform matroids, based on which we can evaluate directly the corresponding Kazhdan-Lusztig polynomial.

Given two positive integers $m$ and $d$, let $U_{m,d}$ be the uniform matroid of rank $d$ on $m+d$ elements.
We are now in the position to give an explicit formula for the inverse Kazhdan-Lusztig polynomial  of uniform matroids.
\begin{thm}\label{main-thm-uniform}
	For any uniform matroid $U_{m,d}$ with $m,d\geq 1$, we have
	$$Q_{U_{m,d}}(t)=\binom{m+d}{d}\sum_{j=0}^{\lfloor (d-1) /2\rfloor}\frac{m(d-2j)}{(m+j)(m+d-j)}\binom{d}{j}t^j.$$
\end{thm}

\begin{proof}
It suffices to show for $0 \leq j \leq \lfloor (d-1) /2\rfloor$, the coefficient of $t^j$ in $\Hat{Q}_{U_{m,d}}(t)$ is
$$(-1)^d\frac{m(d-2j)}{(m+j)(m+d-j)}\binom{m+d}{d}\binom{d}{j}.$$
Applying
equation \eqref{inv_KL_equa_2}
to uniform matroids, we get that
$$\Hat{Q}_{U_{m,d}}(t)
	=\sum_{F\in L(U_{m,d})}t^{\mathrm{rk} \, (U_{m,d})_F}\Hat{Q}_{(U_{m,d})_F}(t^{-1}) \cdot\chi_{(U_{m,d})^F}(t).
$$

Given  $0 \leq i <d$, let $F$ denote any flat of  $U_{m,d}$ with $\mathrm{rk} \, (U_{m,d})_F=i$. It is obvious that the localization $(U_{m,d})_F\cong B_i$ and the contraction $(U_{m,d})^F\cong U_{m,d-i}$,  where  $B_0$ is the unique matroid of rank $0$.
Moreover, we note that  there are $\binom{m+d}{i}$ such flats.
For $i=d$, there is only one flat of rank $d$ which is in fact the ground set $E$.
Hence, it follows that
\begin{align*}
		\Hat{Q}_{U_{m,d}}(t)=\sum_{i=0}^{d-1}\binom{m+d}{i} t^i\Hat{Q}_{B_i}(t^{-1}) \cdot \chi_{U_{m,d-i}}(t)+t^d\Hat{Q}_{U_{m,d}}(t^{-1}),
\end{align*}
which is equivalent to
$$\Hat{Q}_{U_{m,d}}(t)-t^d\Hat{Q}_{U_{m,d}}(t^{-1})=\sum_{i=0}^{d-1}\binom{m+d}{i} t^i\Hat{Q}_{B_i}(t^{-1}) \cdot \chi_{U_{m,d-i}}(t).$$
For $0\leq i \leq d-1$,  from Proposition \ref{Bollean_IKL} it follows that
$$\Hat{Q}_{B_i}(t)=(-1)^i.$$
Meanwhile, it is  known that
\begin{align*}
\chi_{U_{m,d-i}}(t)
	=\sum_{j=0}^{d-i-1}(-1)^j\binom{m+d-i}{j}(t^{d-i-j}-1);
\end{align*}
see \cite[p.121]{White_Neil}.
Hence, we obtain
\begin{align*}
\Hat{Q}_{U_{m,d}}(t)-t^d\Hat{Q}_{U_{m,d}}(t^{-1})
     &=\sum_{i=0}^{d-1} \left((-1)^i\binom{m+d}{i} t^i
	      \times \sum_{j=0}^{d-i-1}(-1)^j\binom{m+d-i}{j}(t^{d-i-j}-1)\right)\\[8pt]
	&=\sum_{i=0}^{d-1}\sum_{j=0}^{d-i-1}(-1)^{i+j}\binom{m+d}{i,j,m+d-i-j}(t^{d-j}-t^i),
	\end{align*}
where $\binom{n}{a_1, a_2, \ldots,a_k}=\frac{n!}{a_1!a_2!\cdots a_k!}$ is the multinomial
coefficient.
Let
\begin{align*}
a_{m,d}&=\sum_{i=0}^{d-1}\sum_{j=0}^{d-i-1}(-1)^{i+j}\binom{m+d}{i,j,m+d-i-j}t^{d-j}\\[5pt]
b_{m,d}&=\sum_{i=0}^{d-1}\sum_{j=0}^{d-i-1}(-1)^{i+j}\binom{m+d}{i,j,m+d-i-j}t^i.
\end{align*}
Thus
\begin{align}\label{eq-abmd}
\Hat{Q}_{U_{m,d}}(t)-t^d\Hat{Q}_{U_{m,d}}(t^{-1})=a_{m,d}-b_{m,d}.
\end{align}
We proceed to reduce the above double summations to single summations. Interchanging the order of the summation for $a_{m,d}$ and then substituting $j$ for $d-j$ lead to
\begin{align*}
a_{m,d}=&\sum_{j=0}^{d-1}\sum_{i=0}^{d-j-1}(-1)^{i+j}\binom{m+d}{i,j,m+d-i-j}t^{d-j}\\[5pt]
	=&\sum_{j=1}^{d}\sum_{i=0}^{j-1}(-1)^{i+d-j}\binom{m+d}{i,d-j,m+j-i}t^{j}.
\end{align*}
Note that
\begin{align}
    \sum_{i=0}^{j-1}(-1)^{i+d-j}\binom{m+d}{i,d-j,m+j-i}=&(-1)^{d-1}\binom{m+d}{d-j} \cdot \sum_{i=0}^{j-1}(-1)^{j-1-i}\binom{m+j}{i}\nonumber\\[5pt]
    =&(-1)^{d-1}\binom{m+d}{d-j}\cdot \sum_{i=0}^{j-1}\binom{-1}{j-1-i}\binom{m+j}{i}\nonumber\\[5pt]
    =&(-1)^{d-1}\binom{m+d}{d-j}\binom{m+j-1}{j-1},\label{eq-conv-form}
    \end{align}
where the last equality is obtained by Vandermonde's convolution.

Now if we set $\binom{n}{-1}=0$ for any $n\geq 0$ by convention, then
\begin{align}\label{eq-amd}
a_{m,d}=&\sum_{j=0}^{d}(-1)^{d-1}\binom{m+d}{d-j}\binom{m+j-1}{j-1}t^{j}.
\end{align}
For $b_{m,d}$, by interchanging the indices $i$ and $j$ and applying \eqref{eq-conv-form}, we obtain
\begin{align}\label{eq-bmd}
b_{m,d}&=\sum_{j=0}^{d-1}\sum_{i=0}^{d-j-1}(-1)^{i+j}\binom{m+d}{i,j,m+d-i-j}t^j \notag \\[5pt]
&=\sum_{j=0}^{d}(-1)^{d-1}\binom{m+d}{j}\binom{m+d-j-1}{d-j-1} t^j.
\end{align}

Combining \eqref{eq-abmd}, \eqref{eq-amd} and \eqref{eq-bmd}, we obtain
\begin{align*}
	\Hat{Q}_{U_{m,d}}(t)&-t^d\Hat{Q}_{U_{m,d}}(t^{-1})\\[5pt]
=(-1)^{d-1}&\sum_{j=0}^{d}\left(\binom{m+d}{d-j}\binom{m+j-1}{j-1}-\binom{m+d}{j}\binom{m+d-j-1}{d-j-1}\right)t^j.
\end{align*}
Note that the degree of $\Hat{Q}_{U_{m,d}}(t)$ is strict less than $\frac{d}{2}$ and hence the degree of lowest term in  $t^d\Hat{Q}_{U_{m,d}}(t^{-1})$ is strict greater than $\frac{d}{2}$. Thus, for any $0\leq j < \frac{d}{2}$, we have
\begin{align*}
[t^j]\Hat{Q}_{U_{m,d}}(t)
&=(-1)^{d-1}\left(\binom{m+d}{d-j}\binom{m+j-1}{j-1}-\binom{m+d}{j}\binom{m+d-j-1}{d-j-1}\right)\\[5pt]
&=(-1)^{d}\frac{m(d-2j)}{(m+j)(m+d-j)}\binom{m+d}{d}\binom{d}{j},
\end{align*}
as desired. This completes the proof.
\end{proof}

Note that the proof of Theorem \ref{main-thm-uniform} only relies on the evaluation of the characteristic polynomials for uniform matroids and the inverse Kazhdan-Lusztig polynomials for boolean matroids. Once Theorem \ref{main-thm-uniform} is established, we find that it is very easy to compute the Kazhdan-Lusztig polynomials for uniform matroids. Note that several formulas for $P_{U_{m,d}}(t)$ have been obtained; see  Gedeon, Proudfoot and Young \cite{Geden2017jcta1}, Gao, Lu, Xie, Yang and Zhang \cite{glxyz_2018arxiv}.  The following result provides a new formula for $P_{U_{m,d}}(t)$.

\begin{cor}\label{uniform_KL-from_Q}
	For any $m,d \geq 1$, we have
\begin{align*}
	P_{U_{m,d}}(t)=\sum_{j=0}^{\lfloor (d-1) /2\rfloor}\sum_{i=0}^{d-1-2j}
	(-1)^{d+1-i}\frac{m(d-i-2j)}{(m+j)(m+d-i-j)}\binom{m+d}{m,i,j,d-i-j}t^j.
\end{align*}
\end{cor}
\proof
Applying Theorem \ref{KL-inverse-KL-2} to $U_{m,d}$, we obtain
$$P_{U_{m,d}}(t)=-\sum_{E \neq F\in L(U_{m,d})}P_{(U_{m,d})_F}(t)\Hat{Q}_{(U_{m,d})^F}(t)
=-\sum_{i=0}^{d-1}\binom{m+d}{i}P_{B_i}(t)\Hat{Q}_{U_{m,d-i}}(t).$$
Note that, for $0 \leq i \leq d-1$,  we have $P_{B_i}(t)=1$; see \cite[Proposition 2.7]{ElisKL2016adv}.
Theorem \ref{main-thm-uniform} tells us that
$$\Hat{Q}_{U_{m,d-i}}(t)=(-1)^{d-i}\sum_{j=0}^{\lfloor (d-i-1) /2\rfloor}\frac{m(d-i-2j)}{(m+j)(m+d-i-j)}\binom{m+d-i}{d-i}\binom{d-i}{j}t^j.$$
It follows that
\begin{align*}
	P_{U_{m,d}}(t)
	&=\sum_{i=0}^{d-1}\binom{m+d}{i}\sum_{j=0}^{\lfloor (d-i-1) /2\rfloor}\frac{(-1)^{d+1-i}m(d-i-2j)}{(m+j)(m+d-i-j)}\binom{m+d-i}{d-i}\binom{d-i}{j}t^j.
\end{align*}
Interchanging the order of the summation, we get the desired result.
\qed

\section{Open problems}\label{sec-4}

Since the matroid Kazhdan-Lusztig polynomials and the inverse Kazhdan-Lusztig polynomials are defined in the same manner, it is natural to ask whether they have some common properties. In this section we shall propose several conjectures for inverse Kazhdan-Lusztig polynomials parallel to those for matroid Kazhdan-Lusztig polynomials.

The first conjecture is concerned with the non-negativity of the coefficients of $Q_{M}(t)$. Elias, Proudfoot and
Wakefield \cite[Conjecture 2.3]{ElisKL2016adv} conjectured that the coefficients of $P_{M}(t)$ for any matroid $M$ are non-negative.
We have the following conjecture.

\begin{conj}\label{conj-non-negative}
For any matroid $M$ the coefficients of $Q_{M}(t)$ are non-negative.
\end{conj}

From Corollary \ref{Bollean_IKL} and Theorem \ref{main-thm-uniform}, it is obvious that Conjecture \ref{conj-non-negative} is true for boolean matroids and uniform matroids. We verified this conjecture  for graphic matroids on $n$ vertices for $n\leq 9$. Note that
Elias, Proudfoot and
Wakefield \cite{ElisKL2016adv} already proved the non=-negativity of $P_{M}(t)$ for any representable matroid $M$, and Braden, Huh, Matherne, Proudfoot, and Wang \cite{bran_huh_2020_arxiv} claimed that this is true for any matroid $M$.

Elias, Proudfoot and Wakefield \cite[Conjecture 2.5]{ElisKL2016adv} also conjectured  that
for any matroid $M$ the coefficients of $P_M(t)$ form a log-concave sequence with no internal zeros.
Recall that a polynomial $f(t)=\sum_{i=0}^n a_it^i$
is said to be log-concave if $a_i^2 \geq a_{i-1}a_{i+1}$ for any $0<i<n$ and
it is said to have no internal zeros if there are no three indices $0\leq i <j <k \leq n$ such
that $a_i,a_k\neq 0$ and $a_j=0$. For the inverse Kazhdan-Lusztig polynomials for matroids, we propose the following conjecture.

\begin{conj}\label{conj-log-concave}
For any matroid $M$ the coefficients of $Q_{M}(t)$ form a log-concave sequence with no internal zeros.
\end{conj}

We verified this conjecture  for graphic matroids on $n$ vertices for $n\leq 9$, and for simple matroids with $n$ elements for $n\leq 9$.
From Corollary \ref{Bollean_IKL}, it is easy to verify Conjecture \ref{conj-log-concave} for boolean matroids.
Next we will show that Conjecture \ref{conj-log-concave} holds for uniform matroids.

\begin{prop}
For any $m,d\geq 1$ the polynomial $Q_{U_{m,d}}(t)$ is log-concave and has no internal zeros.
\end{prop}

\proof Suppose that
$$Q_{U_{m,d}}(t):=\sum_{j=0}^{\lfloor (d-1) /2\rfloor}c_{m,d}^jt^j.$$
By Theorem \ref{main-thm-uniform} we see that, for $0 \leq j \leq \lfloor (d-1) /2\rfloor$,
$$c_{m,d}^j=\frac{m(d-2j)}{(m+j)(m+d-j)}\binom{m+d}{d}\binom{d}{j},$$
which can be rewritten as
$$c_{m,d}^j=\frac{d-2j}{j!(d-j)!} \times \frac{1}{(m+j)(m+d-j)} \times \frac{(m+d)!}{(m-1)!}.$$
Note that the last  factor on the right hand side is independent of $j$. Denote the first factor and the second factor by $a_j$ and $b_j$ respectively, namely,
\begin{align*}
a_j:&=\frac{d-2j}{j!(d-j)!},\qquad b_j:=\frac{1}{(m+j)(m+d-j)}.
\end{align*}
On one hand, for $0< j < \lfloor (d-1) /2\rfloor$, we have
\begin{align*}
\frac{a_j^2}{a_{j-1}a_{j+1}}
&=\frac{(j+1)(d-j+1)(d-2j)^2}{j(d-j)(d-2j-2)(d-2j+2)}.
\end{align*}
On the other hand, for $0< j < \lfloor (d-1) /2\rfloor$, we have
\begin{align*}
\frac{b_j^2}{b_{j-1}b_{j+1}}&= \frac{(m+j-1)(m+j+1)(m+d-j-1)(m+d-j+1)}{(m+j)^2(m+d-j)^2} \\[5pt]
&=\Big(1-\frac{1}{(m+j)^2}\Big)\Big(1-\frac{1}{(m+d-j)^2}\Big)\\[5pt]
&\geq \Big(1-\frac{1}{(1+j)^2}\Big)\Big(1-\frac{1}{(1+d-j)^2}\Big)\\[5pt]
&=\frac{j(j+2)(d-j)(d-j+2)}{(j+1)^2(d-j+1)^2}.
\end{align*}
Thus
\begin{align*}
\frac{(c_{m,d}^{j})^2}{c_{m,d}^{j-1} c_{m,d}^{j+1}}
&=\frac{a_j^2}{a_{j-1}a_{j+1}} \times \frac{b_j^2}{b_{j-1}b_{j+1}}\\[6pt]
&\geq
\frac{j+2}{j+1} \times \frac{d-j+2}{d-j+1} \times \frac{(d-2j)^2}{(d-2j-2)(d-2j+2)}
\geq 1,
\end{align*}
as desired.
\qed

Gedeon, Proudfoot and Young \cite[Conjecture 3.4]{Gedeon_proud_2017sem} conjectured that for any matroid M, all roots of $P_M(t)$ lie on the negative real axis.
However,  computer experiments showed that the inverse Kazhdan-Lusztig polynomials do not always have real zeros. %

\noindent{\bf Acknowledgements.} 
The  authors are very grateful to Arthur L.B. Yang for his helpful comments  on the improvement of this paper.
The first author is supported by the National Science Foundation of China (No.11801447).
The second author is supported  by the National Science Foundation of China (No.11901431).

\end{document}